\documentclass[12pt]{article}

\usepackage[bookmarks=false]{hyperref}
\usepackage{amsfonts, amsmath, euscript, amsthm, amssymb, latexsym}
\usepackage[cp1251]{inputenc}
\usepackage[T2A]{fontenc}
\usepackage[russian]{babel}

\title{Неразрешимое итеративное \\ пропозициональное исчисление}
\author{Боков Г.В.}

\begin{document}

\maketitle

\begin{abstract}
В данной работе рассматриваются итеративные пропозициональные исчисления, представляющие собой конечные множества пропозициональных формул вместе с операцией \emph{modus ponens} и операцией суперпозиции, заданной множеством операций Мальцева. Для таких исчислений изучается вопрос разрешимости проблемы выводимости формул. В работе построено неразрешимое итеративное пропозициональное исчисление, аксиомы которого зависят от трех переменных. Вывод формул в данном исчислении моделирует процесс решение проблемы соответствий Поста. В частности, в работе доказано, что общая проблема выразимости для итеративных пропозициональных исчислений алгоритмически неразрешима.
\end{abstract}

\noindent \textbf{Ключевые слова:} Итеративное пропозициональное исчисление, проблема выводимости, проблема выразимости, проблема соответствий Поста.

\newtheorem{theorem}{Теорема}
\newtheorem{lemma}{Лемма}
\newtheorem{corollary}{Следствие}
\newtheorem*{utv}{Утверждение}
\theoremstyle{definition}
\newtheorem{definition}{Определение}

\section*{Введение}

Пропозициональное исчисление в общем виде представляет собой пару --- конечное множество пропозициональных формул в некоторой сигнатуре и множество операций над этими формулами. Вопрос о разрешимости таких исчислений впервые был поставлен Тарским~\cite{Sinaceur:2000} в 1946 году.

В классическом подходе в качестве операций вывода в пропозициональных исчислениях выступают операция \emph{modus ponens} (из формул $A$ и $A \to B$ выводима формула $B$) и операция \emph{подстановки} (из формулы $A(x)$ выводима формула $A(B)$ для любой формулы $B$). Существование неразрешимого пропозиционального исчисления, а также алгоритмическая неразрешимость многих проблем для классического исчисления высказываний была впервые установлена в 1949 году Линиалом и Постом~\cite{LinialPost:49:RUD}. Аналогичные результаты для интуиционистского исчисления высказываний были получены Кузнецовым~\cite{Kuznetsov:63:UGP} в 1963 году. Исторический обзор корпуса алгоритмически неразрешимых проблем для классических пропозициональных исчислений можно найти в~\cite{Bokov:2015:URA,Bokov:2015:UPPC}. В частности, в~\cite{Bokov:2015:UPPC} приведен исторический обзор методов доказательства неразрешимых свойств таких исчислений.

Анализ моделирования алгоритмически неразрешимых проблем с помощью классических пропозициональных исчислений~\cite{Bokov:2015:UPPC} показал, что оно по большому счету основывается на наличии в исчислениях операции подстановки. Без операции подстановки не удается смоделировать ни одну алгоритмически неразрешимую проблему. С другой стороны, если рассмотреть функциональные системы конечнозначных функций, где вместо операции подстановки используется более слабая \emph{операция суперпозиции}, обычно задаваемая совокупностью операций Мальцева~\cite{Kudryavcev:82:FS}, то нельзя не обратить внимание на то, что большинство проблем для этих систем, наоборот, алгоритмически разрешимы. Эти наблюдения заставляют задуматься о необходимости использования операции подстановки вместо более слабой операции суперпозиции.

Как отмечает Циткин в~\cite{Citkin:2008}, в 1965 году Кузнецов~\cite{Kuznetsov:65:ASS} впервые ввел в рассмотрение операцию \emph{слабой подстановки} (из формул $A(x)$ и $B$ выводима формула $A(B)$) как правило вывода в исчислениях. В~\cite{Kuznetsov:71:FVSL,Kuznetsov:79:SON} данное правило и правило замены эквивалентным использовались для определения выразимости формул в той или иной логике относительной некоторой системы формул. Поскольку слабую подстановку можно представить в виде конечной последовательности операций Мальцева, то в~\cite{Bokov:2014:IPC} данную операцию было принято назвать операцией суперпозиции формул, а исчисления, в которых вместо операции подстановки используется операция суперпозиции, --- итеративными, ввиду их схожести с итеративными алгебрами Поста, введенными Мальцевым~\cite{Malcev:66:IAMP}.

В данной работе будет показано, что ослабление операции подстановки не позволяет полностью избавиться от неразрешимости исчислений. В частности, будет построено итеративное пропозициональное исчисление с неразрешимой проблемой выводимости формул, что доказывает алгоритмическую неразрешимость общей проблемы выразимости для таких исчислений.

\section*{Определения и основные результаты}

Для начала напомним некоторые понятия. Рассмотрим язык, состоящий из счетного множества пропозициональных переменных $\mathcal{V}$ и конечное множество логических связок $\Sigma$, которое будем называть сигнатурой. Буквами $x, y, p$ будем обозначать переменные. Как правило, логические связки унарные или бинарные, например, $\neg$, $\vee$, $\wedge$ или $\to$.

\emph{Пропозициональные формулы} или $\Sigma$-\emph{формулы} строятся из логических связок $\Sigma$ и переменных $\mathcal{V}$ обычным образом. Например, следующие обозначения
\begin{equation*}
  x, \quad \neg A, \quad (A \vee B), \quad (A \wedge B), \quad (A \to B)
\end{equation*}
являются формулами в сигнатуре $\{\neg,\ \vee,\ \wedge,\ \to\}$. Заглавные буквы $A, B, C$ будут использоваться для обозначения формул. Далее условимся опускать внешние скобки, а также скобки, однозначно восстанавливаемые из частичного порядка логических связок.

\emph{Итеративное пропозициональное исчисление} $P$ над множеством логических связок $\Sigma$ это пара, состоящая из конечного множества $\Sigma$-формул $P$, называемых \emph{аксиомами}, и двух правил вывода:

1) \emph{modus ponens}
\begin{equation*}
  A, A \to B \vdash B;
\end{equation*}

2) \emph{суперпозиция} (совокупность операций Мальцева)
\begin{equation*}
  A(x), B \vdash A(B).
\end{equation*}

Обозначим через $[P]$ множество выводимых (или доказуемых) формул исчисления $P$. \emph{Вывод} в $P$ из аксиом с помощью правил вывода определяется обычным образом. Выводимость формулы $A$ из $P$ будем обозначать через $P \vdash A$.

Исчисление $\mathcal{P}$ будем называть \emph{разрешимым}, если существует алгоритм, который по произвольной формуле $A$ отвечает на вопрос: $P \vdash A$? Основным результатом данной работы является следующая теорема.

\begin{theorem} \label{T:main}
Существует неразрешимое итеративное пропозициональное исчисление.
\end{theorem}

Определим на множестве всех пропозициональных исчислений предпорядок. Будем писать $P_1 \leq P_2$ (или $P_2 \geq P_1$), если каждая выводимая в $P_1$ формула также выводима в $P_2$, т.е. $[P_1] \subseteq [P_2]$. \emph{Проблема выразимости} для пропозициональных исчислений состоит в следующем: по двум исчислениям $P_1$ и $P_2$ требуется ответить на вопрос $P_1 \leq P_2$? Как следствие из Теоремы~\ref{T:main} имеем.

\begin{corollary}
Проблема выразимости для итеративных пропозициональных исчислений алгоритмически неразрешима.
\end{corollary}

\section*{Доказательство основного результата}

Прежде чем доказывать основной результат, мы напомним проблему соответствий Поста. Долее мы закодируем слова конечного алфавита пропозициональными формулами и формально докажем сведения проблемы соответствий Поста к проблеме вывода формул в построенном итеративном исчислении.

\subsection*{Проблема соответствий Поста}

Рассмотрим конечный алфавит $\mathcal{A}$, содержащий по крайней мере два символа. Обозначим через $\mathcal{A}^+$ множество непустых слов в алфавите $\mathcal{A}$. \emph{Проблема соответствий Поста}~\cite{Post:46:VRUP} для алфавита $\mathcal{A}$ состоит в следующем. Для последовательности $\Pi$ пар непустых слов в алфавите $\mathcal{A}$
\begin{equation*}
  (\alpha_{1}, \beta_{1}), (\alpha_{2}, \beta_{2}), \ldots, (\alpha_{\mu}, \beta_{\mu})
\end{equation*}
требуется определить, существуют ли такое натуральное число $N \geq 1$ и такие индексы $i_1, \ldots, i_N \in \{1, \ldots, \mu\}$, что выполнено тождество
\begin{equation*}
  \alpha_{i_1} \alpha_{i_2} \ldots \alpha_{i_N} = \beta_{i_1} \beta_{i_2} \ldots \beta_{i_N}.
\end{equation*}
Факт существования такого решения будем обозначать через $\Pi\downarrow$.

Для наглядности рассмотрим следующий пример. Пусть дан алфавит $\{a, b\}$ и три пары непустых слов из данного алфавита
\begin{equation*}
  (\alpha_{1}, \beta_{1}), (\alpha_{2}, \beta_{2}), (\alpha_{3}, \beta_{3}),
\end{equation*}
где $\alpha_{1} = a$, $\alpha_{2} = ab$, $\alpha_{3} = bba$ и $\beta_{1} = baa$, $\beta_{2} = aa$, $\beta_{3} = bb$. Тогда одним из решений проблемы соответствий Поста для данной последовательности будет $N = 4$ и индексы $i_1 = 3$, $i_2 = 2$, $i_3 = 3$, $i_4 = 1$:
\begin{equation*}
  \alpha_{3} \alpha_{2} \alpha_{3} \alpha_{1} = bba + ab + bba + a = bbaabbbaa = bb + aa + bb + baa = \beta_{3} \beta_{2} \beta_{3} \beta_{1}.
\end{equation*}

Поскольку далее мы будем сводить данную проблему к проблеме выводимости формул в итеративных исчислениях, то нам будет удобно определить аналог вывода решения из последовательности пар слов $\Pi$. Для начала введем несколько вспомогательных конструкций.

Рассмотрим две пары $(\alpha, \beta)$ и $(\xi, \zeta)$ слов в алфавите $\mathcal{A}$. Будем говорить, что пара $(\alpha, \beta)$ \emph{элементарно выводима} в $\Pi$ из пары $(\xi, \zeta)$, если найдется такое $i \in \{1, \ldots, \mu\}$, что
\begin{equation*}
  \alpha = \alpha_i \xi \ \text{ и } \ \beta = \beta_i \zeta.
\end{equation*}
Факт элементарной выводимости $(\alpha, \beta)$ из $(\xi, \zeta)$ будем записывать как
\begin{equation*}
  (\xi, \zeta) \stackrel{\Pi}{\longrightarrow } (\alpha, \beta).
\end{equation*}
Расширим понятие выводимости пар слов по транзитивности. Будем назовем пару $(\alpha, \beta)$ \emph{выводимой} в $\Pi$ из пары $(\xi, \zeta)$ и записывать это как
\begin{equation*}
  (\xi, \zeta) \stackrel{\Pi}{\Longrightarrow } (\alpha, \beta),
\end{equation*}
если существует последовательность пар слов $(\xi{1}, \zeta{1}), \ldots, (\xi{n}, \zeta{n})$ такая, что
\begin{enumerate}
  \item $\xi_1 = \xi$ и $\zeta_1 = \zeta$,
  \item $\xi_n = \alpha$ и $\zeta_n = \beta$,
  \item $(\xi_i, \zeta_i) \stackrel{\Pi}{\longrightarrow } (\xi_{i+1}, \zeta_{i+1})$ для всех $1 \leq i \leq n-1$.
\end{enumerate}
Последовательность пар $(\xi{1}, \zeta{1}), \ldots, (\xi{n}, \zeta{n})$ с данными свойствами будем называть \emph{выводом} пары $(\alpha, \beta)$ из пары $(\xi, \zeta)$ в $\Pi$. По определению будем считать, что $(\alpha, \beta) \stackrel{\Pi}{\Longrightarrow } (\alpha, \beta)$ для любой пары слов $(\alpha, \beta)$.

Будем говорить, что пара слов $(\alpha, \beta)$ в алфавите $\mathcal{A}$ \emph{выводима} из $\Pi$ и записывать это как
\begin{equation*}
  \Pi~\vdash~(\alpha, \beta),
\end{equation*}
если $(\alpha, \beta)$ выводима в $\Pi$ из $(\varepsilon, \varepsilon)$, т.е.
\begin{equation*}
  (\varepsilon, \varepsilon) \stackrel{\Pi}{\Longrightarrow } (\alpha, \beta),
\end{equation*}
где $\varepsilon$ --- это пустое слово.

Чтобы провести полную аналогию с выводом формул в пропозициональных исчислениях, определим оператор замыкания множества пар слов $\Pi$. Обозначим через $[\Pi]$ множество всех пар слов, выводимых из $\Pi$:
\begin{equation*}
  [\Pi] := \{(\alpha, \beta) \in \mathcal{A}^+ \times \mathcal{A}^+ \mid \Pi~\vdash~(\alpha, \beta)\}.
\end{equation*}
Несложно убедиться, что $[\cdot]$ является оператором замыкания.

В 1946 году Пост доказал, что проблема соответствий Поста алгоритмически неразрешима.

\begin{theorem}[Пост]
Не существует алгоритма, решающего проблему соответствий Поста для алфавита $\mathcal{A}$.
\end{theorem}

\noindent Он нашел эффективный способ задания однородных систем продукций Поста~\cite{Post:43:FRCP} последовательностями пар непустых слов $\Pi$ в двубуквенном алфавите. Поскольку существует система однородных продукций Поста с алгоритмически неразрешимой проблемой остановки~\cite{Malcev:65:ARF}, то тем самым доказано существование множества пар $\Pi$, для которого множество выводимых пар $[\Pi]$ неразрешимо, т.е. не существует алгоритма, которой по произвольной паре слов $(\alpha, \beta)$  отвечал бы на вопрос, выводима ли $(\alpha, \beta)$ из множества пар $\Pi$, т.е. $\Pi~\vdash~(\alpha, \beta)$. Таким образом, верна следующая теорема, которая понадобится нам в дальнейшем.

\begin{theorem} \label{T:PCP}
Существует множество непустых пар слов $\Pi$ в алфавите $\mathcal{A}$, для которого множество $[\Pi]$ неразрешимо.
\end{theorem}

\subsection*{Кодирование букв и слов формулами}

Рассмотрим конечный алфавит $\mathcal{A} = \{a_1, \ldots, a_m\}$. Мы будем кодировать буквы и слова алфавита $\mathcal{A}$ $\{\to\}$-формулами от двух переменных. Для этого мы фиксируем уникальное переменное $p \in \mathcal{V}$ и введем обозначение
\begin{equation*}
  x \to_i x := ((x \to \underbrace{x) \to \ldots \to x)}_{i} \to x
\end{equation*}
для каждого $i \geq 1$. Таким образом, $x \to_1 x$ --- это формула $(x \to x) \to x$ и $x \to_2 x$ --- это формула $((x \to x) \to x) \to x$.

\emph{Кодом буквы} $a_i$, $1 \leq i \leq m$, будем называть формулу:
\begin{equation*}
  \overline{a}_i[x] := x \to \bigl( (p \to_i p) \to (p \to p) \bigr).
\end{equation*}
Далее иногда будем опускать зависимость кода от переменного $x$ и записывать для краткости $\overline{a}_i$ вместо $\overline{a}_i[x]$.

Формулы $A$ и $B$ будем называть \emph{совместными}, если существуют такие подстановки $\sigma$ и $\pi$ формул вместо переменных, что
\begin{equation*}
  \sigma A = \pi B.
\end{equation*}
В дальнейшем нам понадобиться следующая лемма.
\begin{lemma} \label{L:IPC:Unifiable}
Формулы $\overline{a}[x]$ и $y \to \overline{b}[x]$ несовместны для любых $a, b \in \mathcal{A}$.
\end{lemma}
\begin{proof}
Если бы формулы $\overline{a}[x]$ и $y \to \overline{b}[x]$ были совместны, то совместными оказались бы формулы $p \to p$ и $(p \to_i p) \to (p \to p)$, что невозможно ни для какого $i \geq 1$. Лемма доказана.
\end{proof}

Определим понятие \emph{кода слова} $\alpha \in \mathcal{A}^+$ индукцией по длине слова $|\alpha|$. Если $|\alpha| = 1$, то $\alpha = a$ для некоторой буквы $a \in \mathcal{A}$ и код $\overline{\alpha}$ слова $\alpha$ совпадает с кодом $\overline{a}$ буквы $a$. Пусть $\alpha = \beta b$ для некоторых $\beta \in \mathcal{A}^+$ и $b \in \mathcal{A}$, тогда код $\overline{\alpha}$ слова $\alpha$ определяется соотношением:
\begin{equation*}
  \overline{\alpha}[x] := \overline{\beta}[\overline{b}[x]].
\end{equation*}
Для кодов слов также будем иногда использовать $\overline{\alpha}$ вместо $\overline{\alpha}[x]$.

Для таким образом определенного кодирования естественным образом задана операция конкатенации кодов слов:
\begin{equation*}
  \overline{\xi \zeta} = \overline{\xi}[\overline{\zeta}].
\end{equation*}
Кодом пустого слова $\varepsilon$ будем считать формулу $\overline{\varepsilon}[x] := x$. Легко видеть, что такое определение кода пустого слова согласовано с тем, что оно является нейтральным элементом относительно операции конкатенации.

\subsection*{Исчисления, моделирующие решение проблемы соответствий Поста}

Каждой конечной последовательности $\Pi$ пар непустых слов в алфавите $\mathcal{A}$
\begin{equation*}
  (\alpha_{1}, \beta_{1}), (\alpha_{2}, \beta_{2}), \ldots, (\alpha_{\mu}, \beta_{\mu})
\end{equation*}
мы сопоставим итеративное исчисление $P_{\Pi}$, состоящее из следующих трех групп аксиом:
\begin{equation*}
  \begin{array}{lll}
    \mathrm{(A_1)} & \overline{a}_i[x] & \forall i \in \{1, \ldots, m\}, \\
    \mathrm{(A_2)} & \left( \overline{\alpha}_j[x] \to \overline{\beta}_j[y] \right) \to (x \to y) & \forall j \in \{1, \ldots, \mu\}, \\
    \mathrm{(A_3)} & \overline{a}_i[x] \to \overline{a}_i[x] & \forall i \in \{1, \ldots, m\}.
  \end{array}
\end{equation*}

Далее мы покажем, что данное множество аксиом позволяет смоделировать процесс решения проблемы соответствия Поста для последовательности пар слов $\Pi$, т.е. существует формула, вывод которой в данном исчислении равносилен существованию решения для $\Pi$. Напомним, что решением проблемы соответствий Поста для последовательности пар слов $\Pi$ является такое натуральное число $N \geq 1$ и такие индексы $i_1, \ldots, i_N \in \{1, \ldots, \mu\}$, для которых выполнено тождество
\begin{equation*}
  \alpha_{i_1} \alpha_{i_2} \ldots \alpha_{i_N} = \beta_{i_1} \beta_{i_2} \ldots \beta_{i_N}.
\end{equation*}
Но для начала введем несколько обозначений и докажем вспомогательные леммы.

\subsection*{Выводимость вычислений}

Первая вспомогательная лемма, которая понадобится нам далее, показывает, что из аксиомы $\mathrm{(A_1)}$ выводим код любого непустого слова в алфавите $\mathcal{A}$.

\begin{lemma} \label{L:IPC:WordCodeDerivation}
$\mathrm{(A_1)}~\vdash~\overline{\alpha}$ для любого $\alpha \in \mathcal{A}^+$.
\end{lemma}
\begin{proof}
Докажем лемму индукцией по длине слова $\alpha$. Пусть $\alpha = a_{i_1} \ldots a_{i_n}$ для некоторого $n \geq 1$. Если $n = 1$, то $\overline{\alpha}$ является аксиомой $\mathrm{(A_1)}$ и, следовательно, $\mathrm{(A_1)}~\vdash~\overline{\alpha}$.

Пусть $n > 1$ и утверждение леммы верно для слова $\gamma = a_{i_1} \ldots a_{i_{n-1}}$, т.е.
\begin{equation*}
  \mathrm{(A_1)}~\vdash~\overline{\gamma}.
\end{equation*}
Так как $\overline{\alpha} = \overline{\gamma}[\overline{a}_{i_n}]$, то с помощью операции суперпозиции из кода $\overline{\gamma}$ слова $\gamma$ и аксиомы $\overline{a}_{i_n}$ выводим код $\overline{\alpha}$ слова $\alpha$. Лемма доказана.
\end{proof}

Следующая лемма показывает, что аксиом $\mathrm{(A_1)}$ и $\mathrm{(A_2)}$ достаточно, чтобы смоделировать элементарный вывод в $\Pi$.
\begin{lemma} \label{L:IPC:ToDerivation}
Если $(\xi, \zeta) \stackrel{\Pi}{\longrightarrow} (\alpha, \beta)$, то $P_{\Pi}, \overline{\alpha} \to \overline{\beta}~\vdash~\overline{\xi} \to \overline{\zeta}$.
\end{lemma}
\begin{proof}
Если $(\xi, \zeta) \stackrel{\Pi}{\longrightarrow} (\alpha, \beta)$, то для некоторого $i \in \{1, \ldots, \mu\}$ выполнено
\begin{equation*}
  \alpha = \alpha_i \xi \ \text{ и } \ \beta = \beta_i \zeta.
\end{equation*}
Следовательно, $\overline{\alpha} = \overline{\alpha}_i [\overline{\xi}]$ и $\overline{\beta} = \overline{\beta}_i [\overline{\zeta}]$.

Согласно лемме~\ref{L:IPC:WordCodeDerivation} из аксиомы $\mathrm{(A_1)}$ выводимы коды $\overline{\xi}$ и $\overline{\zeta}$ слов $\xi$ и $\zeta$ соответственно. Поэтому с помощью операции суперпозиции из аксиомы $\mathrm{(A_2)}$ выводима формула
\begin{equation*}
  \left( \overline{\alpha}_i[\overline{\xi}] \to \overline{\beta}_i[\overline{\zeta}] \right) \to (\overline{\xi} \to \overline{\zeta}).
\end{equation*}
Тогда непосредственной проверкой  убеждаемся, что с помощью операции \emph{modus ponens} из данной формулы и формулы $\overline{\alpha} \to \overline{\beta}$ выводима формула $\overline{\xi} \to \overline{\zeta}$. Лемма доказана.
\end{proof}

Используя индуктивное определение вывода в $\Pi$ и лемму~\ref{L:IPC:ToDerivation}, несложно убедиться, что исчисление $P_{\Pi}$ позволяет моделировать любой вывод в $\Pi$.
\begin{corollary} \label{C:IPC:ToDerivation}
Если $(\xi, \zeta) \stackrel{\Pi}{\Longrightarrow} (\alpha, \beta)$, то $P_{\Pi}, (\overline{\alpha} \to \overline{\beta})~\vdash~(\overline{\xi} \to \overline{\zeta})$.
\end{corollary}

\subsection*{Вычисление выводимых формул}

Для произвольной пары слов $(\alpha,\beta)$ в алфавите $\mathcal{A}$ обозначим через $P_{(\alpha,\beta)}$ множество кодов пар слов, из которых выводима пара $(\alpha,\beta)$:
\begin{equation*}
  P_{(\alpha,\beta)} := \{\overline{\xi} \to \overline{\zeta} \mid (\xi, \zeta) \stackrel{\Pi}{\Longrightarrow} (\alpha, \beta)\}.
\end{equation*}
Ясно, что $\overline{\alpha} \to \overline{\beta} \in P_{(\alpha,\beta)}$.

Для произвольной формулы $A$ обозначим через $A^*$ множество подстановочных вариантов формулы $A$, т.е.
\begin{equation*}
  A^* := \{\sigma A \mid \sigma \text{ --- подстановка формул вместо переменных}\}.
\end{equation*}
Для произвольного множества формул $M$ положим
\begin{equation*}
  M^* := \bigcup\limits_{A \in M} A^*.
\end{equation*}
Также определим множество
\begin{equation*}
  P_{\mathcal{V}} := \{(x \to_i x) \to (x \to x) \mid i \geq 1,\ x \in \mathcal{V}\}.
\end{equation*}
Следующая лемма описывает множество выводимых в исчислении $P_{\Pi}$ формул.

\begin{lemma} \label{L:IPC:FromDerivation}
$[P_{\Pi} \cup \{\overline{\alpha} \to \overline{\beta}\}] \subseteq P_{\Pi}^* \cup P_{(\alpha,\beta)}^* \cup P_{\mathcal{V}}^* \cup \mathcal{V}$.
\end{lemma}
\begin{proof}
Доказательство будем вести индукцией по длине вывода $n$. Если $n = 0$, то имеют место включения
\begin{equation*}
  P_{\Pi} \subseteq P_{\Pi}^* \quad \text{ и } \quad \overline{\alpha} \to \overline{\beta} \in P_{(\alpha,\beta)}^*.
\end{equation*}
Пусть утверждение леммы верно для $n \geq 1$, докажем его для $n+1$. Поскольку правая часть включения замкнута относительно операции суперпозиции, то достаточно рассмотреть только случай применения операции \emph{modus ponens}. Рассмотрим произвольную формулу $B$, длина вывода которой равна $n+1$, и пусть формулы $A$ и $A \to B$ имеют вывод в $P_{\Pi} \cup \{\overline{\alpha} \to \overline{\beta}\}$, длина которого не превосходит $n$ для некоторой формулы $A$.

Если $A \to B \in P_{(\alpha,\beta)}^*$, то формула $B$ есть подстановочный вариант кода некоторого слова в алфавите $\mathcal{A}$. Несложно убедиться, что данное слово не может быть пустым, т.к. в противном случае длина вывода $B$ была бы меньше $n+1$, что противоречит выбору формулы $B$. Следовательно, $B$ является подстановочным вариантом аксиомы $\mathrm{(A_1)}$.

Поскольку $A$ не является формулой вида $(x \to_i x)$ ни для какого $i \geq 1$, то $A \to B \notin P_{\mathcal{V}}^* \cup \mathcal{V}$. Поэтому $A \to B \in P_{\Pi}^*$ и остается рассмотреть только следующие три случая:

\begin{description}
  \item[Случай 1.] $A \to B$ --- это подстановочный вариант аксиомы $\mathrm{(A_1)}$, тогда $B \in P_{\mathcal{V}}^*$.

  \item[Случай 2.] $A \to B$ --- это подстановочный вариант аксиомы $\mathrm{(A_2)}$. Несложной проверкой убеждаемся, что $A \in P_{(\alpha,\beta)}^*$. Поэтому $A$ является подстановочным вариантом формулы
      \begin{equation*}
        \overline{\alpha}_i[\overline{\xi}] \to \overline{\beta}_i[\overline{\zeta}]
      \end{equation*}
      для некоторых слов $\xi, \zeta \in \mathcal{A}^*$ и $i \in \{1, \ldots, \mu\}$ таких, что
      \begin{equation*}
        (\alpha_i\xi, \beta_i\zeta) \stackrel{\Pi}{\Longrightarrow} (\alpha, \beta).
      \end{equation*}
      Тогда $B$ является подстановочным вариантом формулы $\overline{\xi} \to \overline{\zeta}$, причем $(\xi, \zeta) \stackrel{\Pi}{\longrightarrow} (\alpha_i\xi, \beta_i\zeta)$. Следовательно, $B \in P_{(\alpha,\beta)}^*$.

  \item[Случай 3.] $A \to B$ --- это подстановочный вариант аксиомы $\mathrm{(A_3)}$, тогда $A = B$, что противоречит выбору формулы $B$.
\end{description}

Данные случаи исключают все возможные варианты. Лемма доказана.
\end{proof}

Как следствие из Леммы~\ref{L:IPC:FromDerivation} имеем.
\begin{corollary} \label{C:IPC:FromDerivation}
Если $P_{\Pi}, \overline{\alpha} \to \overline{\beta}~\vdash~\overline{\xi} \to \overline{\zeta}$, то $(\xi, \zeta) \stackrel{\Pi}{\Longrightarrow} (\alpha, \beta)$.
\end{corollary}
\begin{proof}
Если $P_{\Pi}, \overline{\alpha} \to \overline{\beta}~\vdash~\overline{\xi} \to \overline{\zeta}$, то согласно Лемме~\ref{L:IPC:FromDerivation} выполнено
\begin{equation*}
  \overline{\xi} \to \overline{\zeta} \in P_{\Pi}^* \cup P_{(\alpha,\beta)}^* \cup P_{\mathcal{V}}^* \cup \mathcal{V}.
\end{equation*}
Если слова $\xi$, $\zeta$ непустые, то $\overline{\xi} \to \overline{\zeta} \notin P_{\Pi}^*$ по Лемме~\ref{L:IPC:Unifiable}. Если хотя бы одно из слов $\xi$, $\zeta$ пустое, то $\overline{\xi} \to \overline{\zeta} \notin P_{\Pi}^*$ выполнено согласно выбранному способу кодирования букв и слов формулами. Легко убедиться, что также выполнено
\begin{equation*}
  \overline{\xi} \to \overline{\zeta} \notin P_{\mathcal{V}}^* \cup \mathcal{V}.
\end{equation*}
Значит, $\overline{\xi} \to \overline{\zeta} \in P_{(\alpha,\beta)}^*$ и, следовательно, $(\xi, \zeta) \stackrel{\Pi}{\Longrightarrow} (\alpha, \beta)$ согласно определению множества $P_{(\alpha,\beta)}$. Следствие доказано.
\end{proof}

\subsection*{Сведение проблемы соответствий Поста}

Объединяя Следствия~\ref{C:IPC:ToDerivation} и~\ref{C:IPC:FromDerivation} мы приходим к следующей ключевой лемме.

\begin{lemma} \label{L:IPC:PCP:Key}
$(\xi, \zeta) \stackrel{\Pi}{\Longrightarrow} (\alpha, \beta) \; \Leftrightarrow \; P_{\Pi}, \overline{\alpha} \to \overline{\beta}~\vdash~\overline{\xi} \to \overline{\zeta}$.
\end{lemma}

Следующая теорема описывает формальное сведение проблемы соответствий Поста для последовательности $\Pi$ к проблеме вывода формул в итеративном исчислении $P_{\Pi}$.

\begin{theorem} \label{T:Reduction}
$\Pi\downarrow$ тогда и только тогда, когда $P_{\Pi}~\vdash~(x \to x)$.
\end{theorem}
\begin{proof}
По определению проблема соответствий Поста для входной последовательности $\Pi$ имеет решение, если
\begin{equation*}
  (\varepsilon,\varepsilon) \stackrel{\Pi}{\Longrightarrow} (\gamma, \gamma)
\end{equation*}
для некоторого непустого слова $\gamma \in \mathcal{A}^+$. По Лемме~\ref{L:IPC:PCP:Key} данное условие равносильно выводимости
\begin{equation*}
  P_{\Pi}, \overline{\gamma} \to \overline{\gamma}~\vdash~x \to x.
\end{equation*}
Поскольку формула $\overline{\gamma} \to \overline{\gamma}$ выводима в $P_{\Pi}$ для любого непустого слова $\gamma \in \mathcal{A}^+$, мы получаем, что проблема соответствий Поста для входной последовательности $\Pi$ имеет решение тогда и только тогда, когда $P_{\Pi}~\vdash~(x \to x)$. Теорема доказана.
\end{proof}

Как следствие из Теорем~\ref{T:PCP} и~\ref{T:Reduction} мы имеем доказательство Теоремы~\ref{T:main}.

\end{document}